\documentclass[10pt]{amsart}

\usepackage[a4paper]{geometry}

\usepackage{graphicx}
\usepackage{amsmath,amsthm, amsfonts, amssymb}

\newtheorem*{theorem*}{Theorem}
\newtheorem{lemma}{Lemma}[section]
\newtheorem{proposition}[lemma]{Proposition}

\theoremstyle{definition}
\newtheorem{definition}[lemma]{Definition}
\theoremstyle{remark}

\newcommand{\dist}{\mbox{dist}}

\newcommand{\Vol}{\mbox{Vol}\, }

\title[Regular nonuniformly hyperbolic systems]{A note on Eckmann-Ruelle's conjecture}

\author{ 
Fernando Jos\'e S\'anchez-Salas}
\address{
Departamento de Matem\'aticas, Facultad Experimental de Ciencias, Universidad del Zulia, Avenida Universidad, Edificio Grano de Oro, Maracaibo, Venezuela
}
\email{fjss@fec.luz.edu.ve}

\date{April 8, 2015}

\subjclass[2010]{37D25, 37D35}

\keywords{Nonuniformly hyperbolic systems, rates of escape, hyperbolic product structure}

\thanks{ This work was partially supported by the Associateship Programme of ICTP}

\begin{document}

\begin{abstract}
We introduce a class of $C^{1+\alpha}$ isolated nonuniformly hyperbolic sets $\Lambda$ for which $\sup_{\mu \in {\mathcal M}_f}\{h(\mu) - \chi^+(\mu)\}$ 
equals the rate of escape from $\Lambda$, where $\chi^+(\mu)$ is the average of the sum of positive Lyapunov exponents counted with their multiplicity.
\end{abstract}

\maketitle

\section{Introduction}\label{sec:introduction}

Let $\Lambda$ be a compact $f$-invariant isolated subset, that is, there exists a neighborhood $U$ containing $\Lambda$, such that 
$\Lambda = \bigcap_{-\infty}^{+\infty}f^n(U)$. The {\em rate of escape of $\Lambda$ from $U$} is defined as
\begin{equation}\label{rate.escape.1}
\rho(U) := \limsup_{m \to +\infty}\dfrac{1}{m}\log(\Vol(U_m)),
\end{equation}
where $$U_m = \{\,x \in U \ :\  f^k(x) \in U \ \text{for} \ k = 0, \cdots , m\,\}$$ is the set of points in $U$ 
which stay in $U$ up-to its first $m$ iterates.

It is well known from the thermodynamics of $C^{1+\alpha}$ Axiom A systems that the rate of escape from a suitable small neighborhood $U$ of a topologically transitive 
isolated uniformly hyperbolic set $\Omega$ equals the topological pressure of the unstable potential $\phi^u = -\log(Jac(Df|E^u)$:
$$
P(f | \Omega,\phi^u) = \rho(U).
$$
See \cite{bowen}. By the variational principle for topological pressure 
$$
P(f | \Omega,\phi^u) = \sup\limits_{\mu \in {\mathcal M}_f}\{h(\mu) - \chi^+(\mu)\} = \rho(U),
$$
where
$$
\chi^+(\mu) = \int\sum_{\chi_i(x) > 0}\chi_i(x)\dim(E_i(x))d\mu(x)
$$ 
the average of the sum of positive Lyapunov exponents counted with their multiplicity. Also by the thermodynamics of $C^{1+\alpha}$ Axiom A systems there always exists a unique ergodic 
Borel probability $\mu = \mu_{\phi^u}$ which is an equilibrium state for $\phi^u$ and therefore
$$
h(\mu) - \chi^+(\mu) = \rho(U).
$$

Numerical evidence and heuristic arguments suggested that this property holds for general isolated compact $f$-invariant subsets (see \cite{bonatti.baladi.schmitt}) leading 
Eckmann and Ruelle to raise the following
\ 
\\
\\
{\bf Conjecture} \ {\em Let $\Lambda \subset M$ be a locally maximal compact $f$-invariant set of a smooth diffeomorphism of a compact Riemannian manifold and suppose that 
there exists an ergodic measure $\mu$ such that 
\begin{equation}\label{generalized.SRB}
h(\mu) - \chi^+(\mu) = \sup\limits_{\nu \in {\mathcal M}_f}\{h(\nu) - \chi^+(\nu)\}.
\end{equation}
Then $h(\mu) - \chi^+(\mu) = \rho(U)$ is the rate of escape of a sufficiently small isolating neighborhood $U$ of $\Lambda$,   
}

See \cite{eckmann.ruelle.1985}. An extreme $\mu$ of the variational equation (\ref{generalized.SRB}) is known as a generalized Sinai-Ruelle-Bowen (SRB) 
measure. See \cite{ruelle.1996}. It is not known how generally the conjecture holds beyond systems satisfying Axiom A. 
Eckmann-Ruelle conjecture has been proved for uniformly partially hyperbolic diffeomorphisms \cite{young.1990}, Julia sets of rational maps of the 
Riemann sphere \cite{gonzalo} and certain billiars and so called open systems in \cite{chernov}, \cite{demers.young}, \cite{demers}, 
\cite{lopes.markarian}. We refer the reader to \cite{bonatti.baladi.schmitt} for a review of the problem.

In \cite{bonatti.baladi.schmitt} Baladi, Bonatti and Schmitt constructed a nonuniformly hyperbolic counterexamples to 
Eckmann-Ruelle's conjecture, that is: compact invariant nonuniformly hyperbolic locally maximal set $\Lambda$ with an isolating neighborhood $U$
and an ergodic (unique) non atomic hyperbolic measure $\mu$ such that $\Lambda = supp \ \mu$,
$$
\rho(U) = 0 \quad\text{and}\quad h(\mu) - \chi^+(\mu) = \sup\limits_{\nu \in {\mathcal M}_f}\left\{\,h(\nu) - \chi^+(\nu)\,\right\} < 0.
$$
The example is made from a plug of the eye-like Bowen example into an uniformly hyperbolic repellor $\Lambda_0$ after blowing up a 
suitable fixed point $x_0 \in \Lambda_0$. Then they insert into the blow up a configuration $\Lambda_1$ of finitely many hyperbolic fixed points with 
saddle connections including a copy of a well known Bowen's example made of two hyperbolic fixed points with saddle connections in a eye-like form. 
However the invariant set formed by the two separatrices in Bowen's example do not support an invariant measure so, in principle, it does not gives a 
counterexample for the Eckmann-Ruelle conjecture. Then they handle the plugin in such way that $\mu_0$ the unique generalized SRB measure for 
$\Lambda_0$ extends to a non atomic ergodic measure $\mu$ with support $\Lambda = \Lambda_0 \cup \Lambda_1$ which is also an SRB measure for $\Lambda$ 
and $h(\mu) - \chi^+(\mu) < 0$.

We will prove that, except for this type of examples, Eckmann-Ruelle's conjecture holds for {\em regular nonuniformly hyperbolic sets}.

\begin{definition}\label{definition.regular.nonuniformly.hyperbolic}
Let $\Lambda$ be a compact $f$-invariant nonuniformly hyperbolic set of a $C^{1+\alpha}$ diffeomorphism of a compact manifold. We say that $\Lambda$ is 
a \emph{regular nonuniformly hyperbolic set} if $\Lambda$ is isolated and there exists an increasing sequence of basic sets $\Lambda_n \subset \Lambda$ such that:
\begin{enumerate}
\item $\Lambda = \overline{\bigcup_n\Lambda_{n}}$ and
\item $U = \bigcup_nU_n$ is an isolating open neighborhood of $\Lambda$, where $U_n \subset U_{n+1}$ is an increasing sequence of open isolating neighborhoods of 
$\Lambda_n$.
\end{enumerate}
\end{definition}

We call such a sequence a \emph{regular exhaustion of $\Lambda$}.
\ 
\\
\\
{\bf Theorem A} \ {\em Let $\Lambda \subset M$ be a regular nonuniformly hyperbolic set of a $C^{1+\alpha}$ diffeomorphism. Then,
\begin{equation}\label{eckmann.ruelle.1}
\sup\limits_{\mu \in {\mathcal M}_f}\{h(\mu) - \chi^+(\mu)\} = \rho(U).
\end{equation}
}
\ 
\\
All known nonuniformly hyperbolic sets which are known to verify the Eckmann-Ruelle conjecture are regular. See  \cite{chernov}, \cite{demers.young}, \cite{demers}, 
\cite{lopes.markarian}. We conclude from Theorem A that if $\Lambda$ is a regular nonuniformly hyperbolic isolated set and it supports a nonatomic 
generalized SRB measure then it satisfies the Eckmann-Ruelle conjecture. Theorem A follows from definition \ref{definition.regular.nonuniformly.hyperbolic} 
and the following
\ 
\\
\\
{\bf Theorem B} \ {\em Let $\Lambda$ be a regular nonuniformly hyperbolic set of a $C^{1+\alpha}$ diffeomorphism of a compact manifold. Then,
\begin{equation}\label{eckmann.ruelle.2}
\sup\limits_{\mu \in {\mathcal M}_f(\Lambda)}\{h(\mu) - \chi^+(\mu)\} = \sup\limits_{\Omega \in {\mathcal H}}\rho(U_{\Omega}),
\end{equation}
where ${\mathcal M}_f(\Lambda)$ is the set of $f$-invariant Borel probabilities in $\Lambda$, $\mathcal{H}$ is the family of basic sets $\Omega \subset \Lambda$ and $\rho(U_{\Omega})$ denote the rate of escape from a small isolating neighborhood 
$U_{\Omega}$ of $\Omega \in {\mathcal H}$.
}
\ 
\\
\begin{proof}[Proof of Theorem A]
Let $\Lambda_n \subset \Lambda$ be a regular exhaustion and $U_n$ the corresponding isolating neighborhoods. We shall prove that:
\begin{equation}\label{theoremA.1}
 \rho(U) = \sup_{n > 0}\rho(U_n).
\end{equation}
First we see that (\ref{theoremA.1}) proves Theorem A. By the Theorem B
$$
\rho(U) = \sup_{n > 0}\rho(U_n) \leq \sup\limits_{\Omega \in {\mathcal H}}\rho(U_{\Omega}) = \sup\limits_{\mu \in {\mathcal M}_f(\Lambda)}\{h(\mu) - \chi^+(\mu)\}.
$$
Therefore, $\rho(U) \leq \sup_{\mu \in {\mathcal M}_f}\{h(\mu) - \chi^+(\mu)\}$. The inequality 
$\rho(U) \geq \sup_{\mu \in {\mathcal M}_f}\{h(\mu) - \chi^+(\mu)\}$ follows from L.S. Young's \cite{young.1990} results.

To prove (\ref{theoremA.1}) we first notice that $U^{(m)} = \bigcup_nU^{(m)}_n$ implies $\Vol(U^{(m)}) = \sup_{n > 0}\Vol{U^{(m)}_n}$ and then, as $\Vol{U^{(m)}_n} \leq \Vol{U^{(m)}_{n+1}}$ 
for every $n > 0$,
\begin{eqnarray*}
 \limsup_{m \to +\infty}\dfrac{1}{m}\log(\Vol{U^{(m)}}) & \geq &  \limsup_{m \to +\infty}\dfrac{1}{m}\log\left(\sup_{n > 0}\Vol{U^{(m)}_n}\right)\\
                                                        &  =   & \limsup_{m \to +\infty}\dfrac{1}{m}\sup_{n > 0}\log(\Vol{U^{(m)}_n)}\\
                                                        &  =   & \inf_{m > 0}\sup_{k \geq m}\dfrac{1}{k}\sup_{n > 0}\log(\Vol{U^{(k)}_n})\\
                                                        &  =   & \inf_{m > 0}\sup_{n > 0}\sup_{k \geq m}\dfrac{1}{k}\log(\Vol{U^{(k)}_n}),\\
\end{eqnarray*}
Let us denote
$$
V^{(m)}_n = \sup_{k \geq m}\dfrac{1}{k}\log(\Vol{U^{(k)}_n}).
$$
Then, $\rho(U_n) = \inf_{m > 0}V^{(m)}_n$. Moreover, let 
$$
V^{(m)} = \sup_{k \geq m}\dfrac{1}{k}\log(\Vol{U^{(k)}}).  
$$
Therefore,
$$
V^{(m)} = \sup_{n > 0}V^{(m)}_n \geq \sup_{n > 0}\inf_{m > 0}V^{(m)}_n = \sup_{n > 0}\rho(U_n), \quad\forall \ m > 0.
$$
Hence 
$$
\rho(U) = \inf_{m > 0}V^{(m)} \geq \sup_{n > 0}\rho(U_n).
$$
Now let $\epsilon > 0$ small and choose, for every $n > 0$ an integer $m_n > 0$ such that
\begin{equation}\label{theoremA.2}
 V^{(m_n)}_n < \inf_{m > 0}V^{(m)}_n + \epsilon \quad\forall \ n > 0.
\end{equation}
Then, for every small $\epsilon > 0$.
\begin{eqnarray*}
 \rho(U) = \inf_{m > 0}V^{(m)} & = & \inf_{m > 0}\sup_{n > 0}V^{(m)}_n \\
                               & \leq & \sup_{n > 0}V^{(m_n)}_n \\
                               & \leq & \sup_{n > 0}\inf_{m > 0}V^{(m)}_n + \epsilon\\
                               &  =   & \sup_{n > 0}\rho(U_n) + \epsilon,
\end{eqnarray*}
and we conclude that
$$
\rho(U) \leq \sup_{n > 0}\rho(U_n).
$$
\end{proof}

To prove of Theorem B we use the following

\begin{proposition}\label{prop.1}
$\Lambda$ be a $C^{1+\alpha}$ nonuniformly hyperbolic compact $f$-invariant set. Suppose in addition that there exists an increasing sequence of basic sets 
$\Lambda_n \subset \Lambda$ such that $\Lambda = \overline{\bigcup_n\Lambda_n}$. Then, for every continuous $\phi$
\begin{equation}\label{main.1}
P(f|\Lambda, \phi) = \sup_{\Omega \in \mathcal{H}}P(f|\Omega,\phi)
\end{equation}
where $ \mathcal{H}$ is the family of basic sets $\Omega \subset \Lambda$. 
\end{proposition}

We proved in \cite{sanchez.salas.2015} that (\ref{main.1}) holds in general for every nonuniformly hyperbolic systems and an special class of 
so called \emph{hyperbolic potentials}. Actually we show that there are nonuniformly hyperbolic systems for which (\ref{main.1}) does not hold for every 
continuous potential. This is the case if the dynamics has isolated hyperbolic orbits as is the case, for example for a diffeomorphism of the sphere having a horseshoe with 
internal tangencies and a repeller point at north pole, as we show in \cite{sanchez.salas.2015}. So proposition \ref{prop.1} is interesting since it 
gives a sufficient condition for the variational equation (\ref{main.1}) to hold for \emph{every continuous function}.  

If $\Lambda$ is the support of a nonatomic hyperbolic measure then there exists a sequence of uniformly hyperbolic sets $\Lambda_n \subset \Lambda$ 
such that $\Lambda = \overline{\bigcup_n\Lambda_n}$. This follows from \cite{luzzatto.sanchez}. Notice however that this is not a sufficient
condition for $\Lambda$ to has a regular exhaustion, as Baladi et al example shows.

Proposition \ref{prop.1} follows from the continuity of topological pressure $P(f|\Lambda,\phi)$ with respect to $\Lambda$ and the existence of an 
exhausting sequence $\Lambda_n$.

\begin{definition}
Let $X$ be a compact metric space. We define the Hausdorff distance of compact subsets $A,B \subset X$ as:
$$
\dist_H(A,B) = \inf\{\epsilon > 0: A \subset B_{\epsilon}, \ B \subset A_{\epsilon}\}
$$
where $A_{\epsilon} := \{x \in X: \dist(x,A) < \epsilon\}$ and $\dist(x,A) = \inf_{y \in A}d(x,y)$.
\end{definition}

Let $f : X \to X$ be a continuous selfmap of a compact metric space. The family $\mathcal{I}$ of compact $f$-invariant subsets is 
a compact metric space with the Hausdorff distance.

\begin{lemma}\label{lemma.1}
Let $\phi$ be continuous and $f : X \to X$ be a continuous selfmap of a compact metric space. Then $\Omega  \mapsto P(f|\Omega,\phi)$ is continuous when 
$\Omega$ varies on the family of compact $f$-invariant subsets with $P(f|\Omega,\phi) < +\infty$.
\end{lemma}
\begin{proof}
$$
P(f|\Omega,\phi,\epsilon) = \limsup_{n \to +\infty}\dfrac{\log{P_n(f,\phi,\epsilon)}}{n}
$$
where
$$
P_n(f,\phi,\epsilon) = \sup_E\sum_{x \in E}e^{S_n\phi(x)},
$$
infimun taken over the familiy of $(\epsilon,n)$-separated sets and $S_n\phi(x) = \sum_{k=0}^{n-1}$. Then, given $\delta > 0$, let 
$0 < \epsilon < \delta$ such that if $d(x,y) < \epsilon$ then $|\phi(x) - \phi(y)| < \delta/2$ and $n > 0$ such that
$$
\exp{n(P(f|\Omega,\phi,\epsilon/2) - \delta/2)} < P_n(f,\phi,\epsilon/2) <  \exp{n(P(f|\Omega,\phi,\epsilon/2) + \delta/2)}.
$$
Let $E \subset \Omega$ be a maximal $(\epsilon/2,n)$-separated set. Then, $\Omega \subset \bigcup_{x \in E}B(x,n,\epsilon/2)$. Choose $\eta > 0$ 
sufficiently small such that $d(x,y) < \eta$ implies $d(f^k(x),f^k(y)) < \epsilon/2$ for $k = 0, \cdots , n-1$ and such that 
$\dist_{H}(\Omega,\Omega') < \eta$ implies $\Omega' \subset \bigcup_{x \in E}B(x,n,\epsilon/2)$. For every $x \in E$ we choose $x' \in \Omega'$ such that
$d(x,x') < \eta$ and call $F$ such set. Then $\Omega' \subset \bigcup_{x' \in F}B(x',n,\epsilon)$ and
$$
e^{-n\delta/2}\sum_{x \in E}e^{S_n\phi(x)} < \sum_{x' \in F}e^{S_n\phi(x')} < e^{n\delta/2}\sum_{x \in E}e^{S_n\phi(x)}.
$$
Then,
$$
Q_n(f|\Omega',\phi,\epsilon) = \inf_F\sum_{x \in F}e^{S_n\phi(x)} < \exp{n(P(f|\Omega,\phi,\epsilon/2) + \delta)}
$$
infimum taken over the family of $(\epsilon,n)$-spanning sets $F \subset \Omega'$ and thus
$$
Q(f|\Omega',\phi,\epsilon) = \limsup_{n \to +\infty}\dfrac{\log{Q_n(f,\phi,\epsilon)}}{n} < P(f|\Omega,\phi,\epsilon/2) + \delta.
$$
Therefore, as $P(f|\Omega) = \sup_{\epsilon > 0}Q(f|\Omega,\phi,\epsilon) = \sup_{\epsilon > 0}P(f|\Omega,\phi,\epsilon)$ we conclude that
$$
P(f|\Omega',\phi) < P(f|\Omega,\phi) + \delta.
$$
By symmetry on $\Omega$ and $\Omega'$ we conclude that $|P(f|\Omega',\phi) - P(f|\Omega,\phi)| < \delta$ for $\dist_{H}(\Omega,\Omega') < \eta$. 

\end{proof}
 
\begin{proof}[Proof of Proposition \ref{prop.1}]
Clearly,
$$
P(f|\Lambda, \phi) \geq \sup_{\Omega \in \mathcal{H}}P(f|\Omega,\phi)
$$
since $P(f|\Lambda, \phi) \geq P(f|\Omega,\phi)$ for every compact $f$-invariant $\Omega \subset \Lambda$. On the other hand, by hypothesis and results in 
\cite{luzzatto.sanchez} there exists an exhaustion of $\Lambda$ by a sequence $\Lambda_n \subset \Lambda$ of basic sets such that
$\dist_H(\Lambda_n,\Lambda) \to 0^+$. Therefore, by the continuity of the pressure $P(f|\Omega,\phi)$ respect to $\Omega$,
$$
P(f|\Lambda_n, \phi) \to P(f|\Lambda, \phi),
$$
thus proving that
$$
P(f|\Lambda, \phi) = \sup_{\Omega \in \mathcal{H}}P(f|\Omega,\phi),
$$
where $\mathcal{H}$ is the family of basic sets in $\Lambda$.
\end{proof}

\begin{proof}[Proof of Theorem B]

Let $\bigwedge^{m}TM$ be the (one dimensional) vector bundle of volume forms over $M$, and $\bigwedge^{m}Df : \bigwedge^{m}TM \to \bigwedge^{m}TM$ 
the fiber map induced by the derivative $Df : TM \to TM$. Then, by Oseledec theorem and Kingman's subadditive ergodic theorem
$$
\Phi^u(x) = \lim_{n \to +\infty}-\dfrac{1}{n}\log|\bigwedge^{m}Df^{n}(x)| = \sup_{n > 0}-\dfrac{1}{n}\log|\bigwedge^{m}Df^{n}(x)| = -\sum_{\chi_i(x) > 0}\chi_i(x)\dim{E_i}(x)
$$
$\mu$-a.e. for every $f$-invariant Borel probability $\mu$.

Therefore,
$$
h(\mu) - \chi^+(\mu) = \sup_{n > 0}\left\{h(\mu) - \int\dfrac{1}{n}\log|\bigwedge^{m}Df^{n}(x)|d\mu(x)\right\},
$$
where
$$
\chi^+(\mu) = \int\sum_{\chi_i(x) > 0}\chi_i(x)\dim{E_i}(x)d\mu(x).
$$
Let $\phi^u = -\log\text{Jac}(Df | E^u)$, the induced volume deformation along unstable manifolds. Then by Theorem C
\begin{equation}\label{eq.1}
\sup_{\mu \in \mathcal{M}_f(\Lambda)}\{h(\mu) - \chi^+(\mu)\} = \sup_{\Omega \in \mathcal{H}}P(f|\Omega,\Phi^u) 
\end{equation}
Indeed,
\begin{eqnarray*}
\sup_{\mu \in \mathcal{M}_f(\Lambda)}\{h(\mu) - \chi^+(\mu)\} & = & \sup_{\mu \in \mathcal{M}_f(\Lambda)}\sup_{n > 0}\left\{h(\mu) - \int\dfrac{1}{n}\log|\bigwedge^{m}Df^{n}(x)|d\mu(x)\right\}\\
                                                              & = & \sup_{n > 0}\sup_{\mu \in \mathcal{M}_f(\Lambda)}\left\{h(\mu) - \int\dfrac{1}{n}\log|\bigwedge^{m}Df^{n}(x)|d\mu(x)\right\}\\
                                                              & = & \sup_{n > 0}P\left(f|\Lambda,-\dfrac{1}{n}\log|\bigwedge^{m}Df^{n}|\right)\\
                                                              & = & \sup_{n > 0}\sup_{\Omega \in \mathcal{H}}P\left(f|\Omega,-\dfrac{1}{n}\log|\bigwedge^{m}Df^{n}|\right) \ \text{(by proposition \ref{prop.1})}\\
                                                              & = & \sup_{\Omega \in \mathcal{H}}\sup_{n > 0}P\left(f|\Omega,-\dfrac{1}{n}\log|\bigwedge^{m}Df^{n}|\right)\\
                                                              & = & \sup_{\Omega \in \mathcal{H}}P(f|\Omega,\Phi^u).
\end{eqnarray*}
On the other hand
$$
\lim_{n \to +\infty}\dfrac{1}{n}\phi^u(f^n(x)) = \Phi^u(x), \quad\quad\mu-a.e.\quad\forall \ \mu \in {\mathcal M}_f(\Omega),
$$
by Oseledec's theorem. Then, by Birkhoff's theorem
$$
\int\phi^u(x){d\mu(x)}  = \int\Phi^u(x)d\mu(x) \quad\quad\forall \ \mu \in {\mathcal M}_f(\Omega),
$$
since $\phi^u| \Omega$ is $\mu$-summable for every $f$-invariant Borel probability $\mu \in \mathcal{M}_f(\Omega)$.
Therefore, for every basic set $\Omega \subset \Lambda$,
$$
P(f|\Omega,\Phi^u) = P(f|\Omega,\phi^u) = \rho(U_{\Omega}),
$$
by the thermodynamics of Axiom A systems. See \cite[Proposition 4.8 (a)]{bowen}. This and equation (\ref{eq.1}) proves (\ref{eckmann.ruelle.2}) and, 
therefore, Theorem B.
\end{proof}

\section{When is regular a nonuniformly hyperbolic set?}

As we mentioned before Baladi et al. example shows that to be the support of nonatomic hyperbolic measure is not a sufficient condition for 
$\Lambda$ to be regular. Indeed let $U_0$ be the isolating neighborhood for $\Lambda_0$. Then $U_0 \subsetneq U$ is contained properly into $U$, the isolating neighborhood 
of $\Lambda$ and as long as $\Lambda_1$ is made up finitely many hyperbolic fixed points into an array of saddle connections and 
$\Lambda = closure \ (\Lambda_0)$ there is no chance to get a regular exhaustion of $\Lambda$.

We would like to conclude this note by proposing a sufficient condition for $\Lambda$ to be a regular nonuniformly hyperbolic set. But first we need 
some definitions.
\begin{definition}
Let $\Omega \subset M$ be a compact subset. We say that $\Omega$ has a \emph{hyperbolic product structure} if there exists two continuous laminations 
$\mathcal{F}^s$ and $\mathcal{F}^u$ of disks such that
\begin{enumerate}
\item the disks $\mathcal{F}^s(x)$ and $\mathcal{F}^u(x)$ have a size uniformly bounded between positive constants $\alpha < \beta$; 
\item $d(f^n(y),f^n(y')) \to 0^+$ (resp. $d(f^{-n}(z),f^{-n}(z')) \to 0^+$) exponentially fast for every $y,y' \in \mathcal{F}^s(x)$ (resp. $z,z' \in \mathcal{F}^u(x)$), for every $x \in \Omega$;
\item $\mathcal{F}^s$ and $\mathcal{F}^u$ are invariant: if $x \in \Omega$ and $f^n(x) \in \Omega$ then $f^n\mathcal{F}^s(x) \subset \mathcal{F}^s(f^n(x))$ and $f^{n}\mathcal{F}^u(x) \supset \mathcal{F}^u(f^n(x))$;
\item for every $x,y \in \Omega$ the disk $\mathcal{F}^s(x)$ intersects transversally $\mathcal{F}^u(y)$ at a unique point with an angle uniformly bounded from below by some constant $\gamma > 0$;
\item $\Omega = \bigcup\mathcal{F}^s \cap \bigcup\mathcal{F}^u$. 
\end{enumerate}
\end{definition}
See \cite[Definition 1]{young.1998}.

\begin{definition}
Let $\mu$ be a hyperbolic measure and $\Lambda = supp \ \mu$. We say that $\mu$ has a \emph{local hyperbolic product structure} if there exists Borel 
positive functions $\alpha,\beta, \gamma: \Lambda \to (0,+\infty)$ such that for every $x \in \Lambda$ there exists an open neighborhood $U_x$ and a 
compact set $\Omega \subset U_x$ with $\mu(\Omega) > 0$ having a hyperbolic product structure with constants $\alpha(x),\beta(x), \gamma(x)$.
\end{definition}
\ 
\\
{\bf Conjecture} \ {\em Let $\Lambda$ be an isolated nonuniformly hyperbolic set of a $C^{1+\alpha}$ diffeomorphism. Suppose that $\Lambda$ is 
the support a hyperbolic measure $\mu$ with positive entropy and local hyperbolic product structure. Then $\Lambda$ is regular. 
}
\ 
\\
\\
All the known systems satisfying Eckmann-Ruelle's conjecture have local hyperbolic product structure.

\end{document}